\documentclass[a4paper, 11pt, twoside, notitlepage]{amsart}
\usepackage{graphicx,xcolor,amssymb} 
\usepackage{amsmath,amscd}
\usepackage{amssymb}
\usepackage{amsthm}
\usepackage{appendix}
\usepackage{comment}
\usepackage{hyperref}
\hypersetup{
linkcolor=black,
citecolor=black
}
\usepackage[
  backend=biber,
  style=alphabetic, 
  doi=true,
  url=false,
  eprint=true
]{biblatex}

\AtEveryBibitem{%
  \iffieldundef{doi}
    {}
    {\clearfield{url}\clearfield{urldate}}%
}
\title[Improved stability of low Fourier modes]{Improved stability of low Fourier modes in inverse problems for potentials}

\dedicatory{Dedicated to Roman Novikov on the occasion of his 60th birthday}

\author{Tony Liimatainen}
\address{Department of Mathematics and Statistics, University of Jyv\"askyl\"a, Jyv\"askyl\"a, Finland}
\email{tony.t.liimatainen@jyu.fi}

\author{Mikko Salo}
\address{Department of Mathematics and Statistics, University of Jyv\"askyl\"a, Jyv\"askyl\"a, Finland}
\email{mikko.j.salo@jyu.fi}

\author{William Trad}
\address{Department of Mathematics and Statistics, University of Jyv\"askyl\"a, Jyv\"askyl\"a, Finland}
\email{william.n.k.trad@jyu.fi}

\newtheorem{theorem}{Theorem}[section]

\newtheorem{lemma}[theorem]{Lemma}
\newtheorem{proposition}[theorem]{Proposition}

\theoremstyle{definition}

\theoremstyle{remark}
\newtheorem{remark}{Remark}[section]

\newcommand{\f}[1]{\footnote{\textcolor{blue}{#1}}}
\newcommand{\R}{{\mathbb R}}

\newcommand {\p} {\partial}

\newcommand{\norm}[1]{\lVert #1 \rVert}

\begin{document}

\maketitle

\begin{abstract}
We prove Lipschitz, sub-H\"older and H\"older stability estimates for recovering the low 
Fourier modes of an unknown potential from the Dirichlet-to-Neumann (DN) map. We study three different cases, depending on the regularity of the difference $q_1-q_2$. 

First, we consider 
\[
(-\Delta - \lambda^2 + q)u=0 \quad \text{in} \quad \Omega\subset \R^n.
\]
We show that the difference 
$q_1-q_2$, assumed to be $M$-bandlimited, can be recovered in a Lipschitz 
stable way from the difference of the corresponding DN maps. This holds 
whenever $\lambda$ is sufficiently large relative to $M^{n/2}$. The 
proof involves real geometrical optics solutions.

Secondly, we consider 
\[
(-\Delta + q)u=0 \quad \text{in} \quad (-\pi,\pi)^n.
\]
We show that the low Fourier 
coefficients of the difference $q_1-q_2$, assumed to be real-analytic and periodic, can be recovered in a 
sub-H\"older stable way from the difference of the corresponding DN maps. The 
number of recoverable Fourier modes grows as the DN maps become closer. 

Finally, we consider the case where the Fourier coefficients of the difference $q_1-q_2$ decay at a super-exponential rate $e^{-c|k|^{n/2}}$. We prove that the low Fourier modes can be recovered with H\"older stability, with the number of recoverable modes tending to infinity as the DN maps become closer.

In all cases the $L^\infty$ potentials themselves do not need to satisfy additional assumptions or belong to a finite dimensional space. The constants in the stability estimates are uniform in the number of recovered Fourier modes.

\end{abstract}

\section{Introduction}

In this article we study stability aspects of inverse problems for elliptic equations, with a focus on understanding which features of an unknown potential lead to more stable recovery. Our work is motivated by the increasing stability phenomenon. Starting with the work of Isakov \cite{isakov_increasing_2011}, it has been understood that in the presence of a large frequency $\lambda$, one can obtain stability estimates that improve as $\lambda$ increases. Most results in this direction establish conditional stability estimates where the modulus of continuity consists of a Hölder part and a logarithmic part, with the latter formally vanishing as the frequency tends to infinity. 

The recent article \cite{kow_increasing_2025} studied the related notion of \emph{increasing resolution}. The key insight is that at large frequencies $\lambda$, certain features of the unknown, such as individual Fourier coefficients, can be recovered in a Lipschitz stable manner, and the number of such stably recoverable features grows with $\lambda$. This was demonstrated in \cite{kow_increasing_2025} for linearized inverse scattering problems by estimating the singular values of the linearized forward operator. Nonlinear inverse problems, however, require new ideas beyond singular value estimates.

In the present work, we study stable recovery of low Fourier modes in the nonlinear inverse problem for a potential under three different regularity assumptions on the potential difference. For bandlimited differences at large frequencies, we obtain Lipschitz stability with increasing resolution. For real-analytic and super-exponentially decaying differences, we obtain sub-Hölder and Hölder stability respectively, with the number of recoverable modes growing as the measurement error decreases. In all cases, only the difference of the potentials needs to satisfy the regularity assumption.

\subsection{Statement of results}

Let $\Omega \subset \mathbb{R}^n$ be a bounded open set with smooth boundary, and let $q \in L^{\infty}(\Omega)$. We mostly work on  $\Omega = (-\pi,\pi)^n$, as this is the natural setting for Fourier expansions. 

We first consider the elliptic boundary value problem
\begin{align}
    \begin{cases}
        (-\Delta-\lambda^2+q)u=0 & \text{in } \Omega, \\ 
        u = f & \text{on } \partial \Omega,
    \end{cases}\label{eq:bvp}
\end{align}
where $\lambda > 0$ is such that the Dirichlet-to-Neumann map 
\[
\Lambda_q(\lambda): H^{1/2}(\partial \Omega) \ni f \mapsto \partial_\nu u|_{\partial \Omega} \in H^{-1/2}(\partial \Omega)
\]
is well-defined. When $\lambda = 0$, we simply write $\Lambda_q$.

Our first result addresses the case where the difference of two potentials is bandlimited. We say that a function $q \in L^{\infty}(\Omega)$ is $M$-\emph{bandlimited} if it can be expressed as
\[
q = \sum_{k \in \mathbb{Z}^n, \, |k| \leq M} q_k e^{ik \cdot x} \Big|_{\Omega}
\]
for some coefficients $q_k \in \mathbb{C}$.

\begin{theorem}[Bandlimited difference]\label{thm_main1}
Let $\Omega = (-\pi, \pi)^n$, $n \geq 2$, and fix $E \geq 1$. There exists a constant $C = C(\Omega, n, E) \geq 1$ such that the following holds. For any $M \geq 1$ and any potentials $q, q' \in L^{\infty}(\Omega)$ satisfying $
\|q\|_{L^{\infty}(\Omega)} \leq E$, $\|q'\|_{L^{\infty}(\Omega)} \leq E$, 
with $q - q'$ being $M$-bandlimited, we have
\[
\sup_{|k| \leq M} |q_k - q_k'| \leq C \|\Lambda_q(\lambda) - \Lambda_{q'}(\lambda)\|_{L^2(\partial \Omega) \to L^2(\partial \Omega)},
\]
whenever $\lambda \geq C M^{n/2}$. 
\end{theorem}

Several remarks are in order. First, the theorem shows that if the frequency $\lambda$ is sufficiently large relative to the bandwidth $M$, then all Fourier modes of the difference can be recovered with a Lipschitz stability constant that is independent of $M$. This is in contrast to typical finite-dimensional Lipschitz stability results, where the Lipschitz constant grows very rapidly as the dimension of the unknown grows. 

Second, we emphasize that only the \emph{difference} $q - q'$ is assumed to be bandlimited. The individual potentials need only be bounded. This relaxes the typical assumption that the unknowns belong to a finite-dimensional subspace.

Third, the estimate in Theorem~\ref{thm_main1} is stated in terms of the norm of $\Lambda_q(\lambda) - \Lambda_{q'}(\lambda)$ on $L^2(\partial \Omega)$ instead of the $H^{1/2}(\partial \Omega) \to H^{-1/2}(\partial \Omega)$ norm. In Lemma~\ref{lemma_dnmap_difference_ltwo} it is proved that the norm in $L^2(\partial \Omega)$ is indeed finite.

Fourth, Theorem~\ref{thm_main1} is stated when $\Omega$ is a cube, which is compatible with the definition of bandlimited potentials given in terms of complex exponentials. We have a similar result for any bounded domain $\Omega$ (see Theorem~\ref{thm_main1_omega}). However, that result is less clean since the constants depend on the norm of the inverse of a Landau--Pollak--Slepian operator which itself depends on $M$.

Theorem~\ref{thm_main1} raises a natural question: can one go beyond bandlimited differences and recover low frequencies of more general potentials? In certain linearized problems, it is known that the low-frequency part of a smooth potential can be Lipschitz stably recovered if the frequency is large enough \cite{kow_increasing_2025}. For nonlinear inverse problems, this would correspond to the existence of stable functionals as discussed in \cite{ALESSANDRINI2005207}.

Our second main result considers potentials whose difference is real-analytic, obtaining sub-Hölder stability for the low Fourier modes. For simplicity, we will only do this in the case $\lambda=0$ and for $\Omega = (-\pi,\pi)^n$. We denote the DN map of
\begin{align*}
    \left\{
    \begin{aligned}
        (-\Delta+q)u &= 0  &&\text{in } \Omega, \\ 
        u &= f  &&\text{on } \partial\Omega,
    \end{aligned}
    \right.
\end{align*}
by $\Lambda_q$. In both Theorems \ref{thm:analytic_case} and \ref{thm:stretched_exponential}, we will assume that $0$ is not a Dirichlet eigenvalue of $-\Delta+q$. Prior to stating the theorems, we additionally make clear that when a function $F$ is real-analytic on the torus, it admits a holomorphic extension to a complex neighborhood (Grauert tube) $\Omega_{\varepsilon}$ of some radius $\varepsilon > 0$, namely $\Omega_{\varepsilon} = \{ z \in \mathbb{C}^n : |\mathrm{Im}(z)| < \varepsilon \}/(2\pi\mathbb{Z})^n$. We denote the $L^\infty$ norm of this holomorphic extension by $\|F\|_{L^\infty(\Omega_{\varepsilon})}$.

\begin{theorem}[Real-analytic difference]\label{thm:analytic_case}
Let $\Omega = (-\pi,\pi)^n$, $n \geq 3$. Let $q, q' \in L^{\infty}(\Omega)$ satisfy $\|q\|_{L^{\infty}(\Omega)}, \|q'\|_{L^{\infty}(\Omega)} \leq E_1$, and assume that $Q = q - q'$ is real-analytic and periodic. Suppose furthermore that $Q$ extends holomorphically to a complex neighborhood $\Omega_\varepsilon$ of radius $\varepsilon > 0$ with $\|Q\|_{L^{\infty}(\Omega_\varepsilon)} \leq E_2$. Let $\delta = \|\Lambda_q - \Lambda_{q'}\|_{L^2(\partial\Omega) \to L^2(\partial\Omega)}$. Then there exist constants $C, \gamma > 0$ depending on $\Omega, n, E_1, E_2$, and a function $M(\delta) > 0$ with $M(\delta) \to \infty$ as $\delta \to 0$, such that for all sufficiently small $\delta$,
\[
\sup_{|k| \leq M(\delta)} |Q_k| \leq C \exp\left(-\sigma \left(\frac{1}{\gamma+\sigma}\log(1/\delta)\right)^{2/n}\right),
\]
where $\sigma = \varepsilon/(e \sqrt{n})$, and $M(\delta)$ satisfies
\[
M(\delta) \geq \left(\frac{1}{\gamma+\sigma}\log(1/\delta)\right)^{2/n}.
\]
\end{theorem}

The stability estimate in Theorem~\ref{thm:analytic_case} is of \emph{sub-Hölder} type. It is worse than any Hölder modulus but better than logarithmic. Indeed, as $\delta \to 0$, the bound in the theorem behaves like $\exp(-c (\log(1/\delta))^{2/n})$. For $n \geq 3$, the exponent $2/n$ lies in $(0,1)$, so $(\log(1/\delta))^{2/n}$ grows slower than $\log(1/\delta)$. Consequently, for any $\alpha, \beta > 0$ and for $\delta$ small enough, 
\[
\delta^\alpha = \exp(-\alpha \log(1/\delta)) \lesssim \exp(-c (\log(1/\delta))^{2/n}) \lesssim (\log(1/\delta))^{-\beta}
\]
The first inequality says that the bound decays slower than any Hölder modulus $\delta^\alpha$. The second inequality follows from the fact that $(\log(1/\delta))^{2/n}$ grows faster than $\log\log(1/\delta)$, so $\exp(-c (\log(1/\delta))^{2/n})$ decays faster than any negative power of $\log(1/\delta)$. Thus the estimate interpolates between Hölder and logarithmic stability.

Theorem~\ref{thm:analytic_case} exhibits a clear dependence on the radius $\varepsilon$ of analyticity. A larger domain of holomorphic extension yields a stronger stability estimate. Moreover, the theorem shows that the smaller the difference of the DN maps, the more Fourier modes can be sub-Hölder stably recovered. In particular, given an a priori bound $\delta_0 > 0$, the Fourier modes up to $M(\delta_0)$ can be recovered with the sub-Hölder estimate  whenever $\delta \leq \delta_0$. This provides a quantitative relation between measurement precision and resolution in the nonlinear inverse problem.

Our third result pushes the regularity assumptions further by considering potentials whose Fourier coefficients decay at a \emph{super-exponential} rate. Specifically, we assume that the Fourier coefficients of $Q = q - q'$ satisfy
\begin{equation}\label{eq:stretched_decay_intro}
|Q_k| \leq C e^{-c |k|^{n/2}}, \qquad k \in \mathbb{Z}^n,
\end{equation}
for some $c > 0$. Since $n \geq 3$, this decay is faster than the exponential decay $e^{-\sigma|k|}$ characteristic of real-analytic functions. 
Any series of the form $\sum_{k} a_k e^{-c |k|^{n/2}} e^{i k \cdot x}$ with bounded coefficients $(a_k)$ satisfies the decay condition.

\begin{theorem}[Super-exponential decay]\label{thm:stretched_exponential}
Let $\Omega = (-\pi,\pi)^n$, $n \geq 3$. Let $q, q' \in L^{\infty}(\Omega)$ with $\|q\|_{L^{\infty}(\Omega)}, \|q'\|_{L^{\infty}(\Omega)} \leq E$, and assume that $Q = q - q'$ is periodic and satisfies \eqref{eq:stretched_decay_intro}. Let $\delta = \|\Lambda_q - \Lambda_{q'}\|_{L^2(\partial\Omega) \to L^2(\partial\Omega)}$. Then there exist constants $C > 0$ and $M(\delta) > 0$ with $M(\delta) \to \infty$ as $\delta \to 0$ such that for all sufficiently small $\delta$,
\[
\sup_{|k| \leq M(\delta)} |Q_k| \leq C \delta^\alpha,
\]
where
\[
\alpha = \frac{c}{\gamma + c}, \qquad M(\delta) = \left( \frac{1}{\gamma + c} \log(1/\delta) \right)^{2/n},
\]
and $\gamma$ is as in Theorem~\ref{thm:analytic_case}.
\end{theorem}

Theorem~\ref{thm:stretched_exponential} shows that the stronger decay of Fourier coefficients upgrades the stability of low Fourier modes from sub-Hölder to Hölder, with the Hölder exponent $\alpha = c/(\gamma + c)$ depending explicitly on the decay rate $c$. As $c \to \infty$, we have $\alpha \to 1$, recovering near-Lipschitz stability. Moreover, the smaller the difference of the DN maps, the more Fourier modes can be recovered with Hölder stability.

\subsection{Relation to previous work}

Let us give some further references to conditional stability estimates related to our context. A logarithmic stability result for the inverse problem for potentials was proved by Alessandrini in \cite{alessandriniStableDeterminationConductivity1988}. There is a substantial literature on such estimates and we refer the reader to \cite{feldman_calderproblem_nodate} and references therein. The fact that higher smoothness of the potential leads to improved logarithmic stability estimates has been studied in detail by Novikov in \cite{Novikov2011}, and effectivized reconstructions at large frequencies are given in \cite{Novikov2009}. See \cite{IsaevNovikov2012, IsaevNovikov2014} and the survey \cite{Novikov2022} for further references in this direction. We also mention \cite{ammari_stability_2013}, which contains related stability results at high frequencies. For complementary instability results, see  \cite{mandache_exponential_2001, koch_instability_2021} and references therein. 

When the unknown conductivity belongs to a finite dimensional space, Lipschitz stability estimates were proved in \cite{ALESSANDRINI2005207} with precise estimates for the Lipschitz constant. More recently, inverse problems with finite measurements have been studied systematically in \cite{alberti2019calderon, alberti2022calderon, alberti2022infinite, alberti2023inverse}. In particular, \cite{alberti2019calderon} showed that for the Calderón problem, if the unknown is a finite linear combination of known basis functions, a finite number of boundary measurements is sufficient for reconstruction, assuming Lipschitz stability estimates for the problem. This was extended to independent measurements in \cite{alberti2022calderon}, and to a large class of inverse problems in \cite{alberti2022infinite}, where it is shown that if Lipschitz stability holds, the same estimate remains valid even with a finite number of measurements. In \cite{garde2026infinite} it was proved that certain (possibly infinite) basis expansions can be recovered Lipschitz stably in the inverse conductivity problem. See also \cite{bourgeois2013, carstea2026} for related Lipschitz and Hölder stability results.

A common feature of many of these works is that Lipschitz stability is achieved by assuming the unknowns belong to a known finite-dimensional subspace. Also, the constant in the stability estimate typically grows exponentially in the dimension of the space, or cannot be estimated. 

In contrast, our results require only the difference $q - q'$ to satisfy the respective regularity assumptions. The individual potentials $q$ and $q'$ are merely required to be bounded. Moreover, in Theorem~\ref{thm:analytic_case} and Theorem~\ref{thm:stretched_exponential}, neither the unknown potentials nor their difference is assumed to belong to any finite-dimensional subspace, yet we obtain quantitative control on finitely many Fourier modes. Finally, in Theorem~\ref{thm_main1} the Lipschitz constant is uniform over spaces of dimension $\lesssim \lambda^{2/n}$, in contrast to the exponential growth in the dimension typical of finite-dimensional results.

\subsection*{Organization of the paper}

Section~\ref{sec_bandlimited} proves Theorem~\ref{thm_main1} for bandlimited differences by using real geometrical optics solutions. Section~\ref{sec_real_analytic} establishes Theorem~\ref{thm:analytic_case} for real-analytic differences via complex geometrical optics solutions and estimates on the Fourier coefficients of periodic analytic functions. Section~\ref{sec_super_exponential} proves Theorem~\ref{thm:stretched_exponential} for super-exponential decay by combining the CGO framework with sharper tail estimates. Appendix~\ref{sec_difference_dn_maps} shows that the difference of the DN maps is a bounded operator on $L^2(\partial\Omega)$, ensuring the norms in our estimates are well-defined.

\subsection*{Acknowledgements}

All authors were partly supported by the Research Council of Finland (Centre of Excellence in Inverse Modelling and Imaging and FAME Flagship, grants 353091 and 359208).

\section{Bandlimited difference} \label{sec_bandlimited}

We use $q,q' \in L^{\infty}(\Omega)$ to denote the potentials of interest. In addition, we use $Q:=q-q'$ to denote the difference of these potentials. The assumption that $q,q'$ are $M$-bandlimited implies that 
\begin{align}\label{eq:bandlimited_def}
    Q=\sum_{|k|\leq M}Q_k e^{ik\cdot x} \text{ where } Q_k \in \mathbb{C}.
\end{align}
We consider the vector 
\[
\vec{Q} = (Q_k)_{|k| \leq M}.
\]
We will prove the following more general version of Theorem \ref{thm_main1}.

\begin{theorem} \label{thm_main1_omega}
    Let $\Omega\subset (-\pi,\pi)^n$, $n \geq 2$, be a bounded open set with smooth boundary, and let $E \geq 1$ be fixed. There is $C = C(\Omega, n, E) \geq 1$ such that for any $M \geq 1$, if $\norm{q}_{L^{\infty}(\Omega)} \leq E$, $\norm{q'}_{L^{\infty}(\Omega)} \leq E$ and $Q = q-q'$ is $M$-bandlimited, 
    then we have  
    \[ 
    \norm{T \vec{Q}}_{\ell^{\infty}}\leq C\|\Lambda_q(\lambda)-\Lambda_{q'}(\lambda)\|_{L^2(\partial \Omega)\rightarrow L^2(\partial \Omega)} + C(M^{n/2}/\lambda) \norm{\vec{Q}}_{\ell^{\infty}},
    \]
    whenever $\lambda \geq \max(M^{n/2}, C)$, where $T$ is a positive definite matrix with entries
    \[
    T_{k,l} = \int_{\Omega} e^{-i(k-l) \cdot x} \,dx, \qquad |k|, |l| \leq M.
    \]

    In particular, if 
    \[
    \lambda \geq 2 C M^{n/2} \norm{T^{-1}}_{\ell^{\infty} \to \ell^{\infty}},
    \]
    then one has the Lipschitz stability estimate 
    \[
    \norm{\vec{Q}}_{\ell^{\infty}} \leq 2 C \norm{T^{-1}}_{\ell^{\infty} \to \ell^{\infty}} \|\Lambda_q(\lambda)-\Lambda_{q'}(\lambda)\|_{L^2(\partial \Omega)\rightarrow L^2(\partial \Omega)}.
    \]
\end{theorem}

The operator $T$ above is essentially a Landau-Pollak-Slepian operator appearing in time-frequency analysis, i.e. 
\[
T = P_M \chi_{\Omega} P_M
\]
where $P_M$ is the orthogonal projection to Fourier series over $|k| \leq M$. If $\Omega = (-\pi,\pi)^n$, then $T$ is the identity operator on $M$-bandlimited functions, which gives Theorem \ref{thm_main1}. However, if $\Omega$ is a strict subset of $(-\pi,\pi)^n$, then $\norm{T^{-1}}$ will grow as $M \to \infty$.

We will make use of the following (real) geometrical optics solutions. We only state this result for sufficiently large $\lambda$, since this is enough for our purposes. The constants will then depend on a bound for $\norm{q}_{L^{\infty}(\Omega)}$.

\begin{proposition} \label{prop_go_solutions}
Let $\Omega\subset \R^n$, $n \geq 2$, be a bounded open set with Lipschitz boundary, and let $q \in L^{\infty}(\Omega)$ satisfy $\norm{q}_{L^{\infty}(\Omega)} \leq E$. There are $C, \lambda_0 > 0$ depending only on $\Omega$, $n$, and $E$ such that for any $\lambda \geq \lambda_0$ and for any $\eta \in \R^n$ with $|\eta| = \lambda$, there is a solution $u \in H^2(\Omega)$ of $(-\Delta-\lambda^2+q)u = 0$ in $\Omega$ having the form 
\[
u = e^{i\eta \cdot x}(1+r),
\]
where 
\[
\lambda \norm{r}_{L^2(\Omega)} + \norm{r}_{H^1(\Omega)} + \norm{r}_{L^2(\partial \Omega)} \leq C.
\]
\end{proposition}
\begin{proof}
We first choose $\tilde{r}$ to be a suitable solution of  
\[
(-\Delta-\lambda^2+q) \tilde{r} = -q e^{i \eta \cdot x} \text{ in $\Omega$}.
\]
Denote by $f$ the extension of the right hand side by zero to $\R^n$. Using notation as in \cite[Chapter 7]{feldman_calderproblem_nodate}, given $\delta > 1/2$ the outgoing resolvent $(-\Delta-\lambda^2)^{-1}: L^2_{\delta} \to H^2_{-\delta}$ satisfies 
\[
\lambda \norm{(-\Delta-\lambda^2)^{-1}}_{L^2_{\delta} \to L^2_{-\delta}} + \norm{(-\Delta-\lambda^2)^{-1}}_{L^2_{\delta} \to H^1_{-\delta}}\leq C_{n,\delta}, \qquad \lambda \geq 1.
\]
Thus $q(-\Delta-\lambda^2)^{-1}$ is bounded on $L^2_{\delta}$ with norm $\leq \tilde{C}_{n,\delta, \Omega} E/\lambda$. If we choose $\lambda_0 = \lambda_0(\Omega, n, E, \delta)$ large enough, then  the norm is $\leq 1/2$ when $\lambda \geq \lambda_0$. It follows that for $\lambda \geq \lambda_0$ the equation 
\[
(-\Delta-\lambda^2+q) \tilde{r} = f
\]
has a solution $\tilde{r} \in H^2_{-\delta}(\R^n)$ that satisfies 
\[
\lambda \norm{\tilde{r}}_{L^2_{-\delta}} + \norm{\nabla \tilde{r}}_{L^2_{-\delta}} \leq 2 C_{n,\delta} \norm{f}_{L^2_{\delta}}.
\]
If we fix e.g.\ $\delta=1$, the right hand side is $\leq C(\Omega,n,E)$.
Since $\Omega$ has Lipschitz boundary, the trace theorem implies that 
\[
\norm{\tilde{r}}_{H^{1/2}(\partial \Omega)} \leq C_{\Omega,n} \norm{\tilde{r}}_{H^1(\Omega)} \leq C(\Omega,n,E).
\]
Now, if we take $r = e^{-i \eta \cdot x} \tilde{r}|_{\Omega}$, then $r$ will satisfy the required properties.
\end{proof}

\begin{proof}[Proof of Theorem \ref{thm_main1_omega}]
We begin with the Alessandrini identity 
\begin{equation} \label{bandlimited_alessandrini}
( (\Lambda_q(\lambda)-\Lambda_{q'}(\lambda)) (u|_{\p \Omega}), u'|_{\p \Omega})_{\p \Omega} = \int_{\Omega} Q u \bar{u}' \,dx
\end{equation}
which is valid for $u, u' \in H^1(\Omega)$ solving $(-\Delta-\lambda^2+q)u = (-\Delta-\lambda^2+q')u' = 0$ \cite{feldman_calderproblem_nodate}. We will choose $u$ and $u'$ to be real geometrical optics solutions. Given $k \in \mathbb{Z}^n$ with $|k| \leq 2 \lambda$, we choose a unit vector $\omega \in \R^n$ with $k \cdot \omega = 0$ as well as real vectors 
\begin{align*}
    \eta = -\frac{k}{2} + \sqrt{\lambda^2-\frac{|k|^2}{4}}\omega, \qquad \eta' = \frac{k}{2} + \sqrt{\lambda^2-\frac{|k|^2}{4}}\omega.
\end{align*}
Then $|\eta| = |\eta'| = \lambda$.

By Proposition \ref{prop_go_solutions}, if $\lambda \geq \lambda_0(\Omega,n,E)$ there exist $H^2(\Omega)$ solutions of $(-\Delta-\lambda^2+q)u = (-\Delta - \lambda^2 + q')u' = 0$ in $\Omega$, having the form 
\[
u = e^{i \eta \cdot x}(1+r), \qquad u' = e^{i \eta' \cdot x}(1+r')
\]
with $\lambda \norm{r}_{L^2(\Omega)} + \norm{r}_{L^2(\p \Omega)} \leq C(\Omega,n,E)$ and similarly for $r'$.

Inserting $u$ and $u'$ in \eqref{bandlimited_alessandrini}, we have 
\begin{multline*}
\left| \int_{\Omega} e^{-ik \cdot x} Q \,dx \right| \leq \|\Lambda_q(\lambda)-\Lambda_{q'}(\lambda)\|_{L^2(\partial \Omega)\rightarrow L^2(\partial \Omega)} \|u\|_{L^2(\partial \Omega)}\|u'\|_{L^2(\partial \Omega)} \\
 + \left| \int_{\Omega} e^{-ik \cdot x} Q (r + r' + r r') \,dx \right|.
\end{multline*}
We will use the notation 
\[
r_k = \int_{\Omega} e^{-ik \cdot x} r(x) \,dx.
\]
We can identify $r$ with its zero extension to $(-\pi, \pi)^n$, and then $r_k$ are just the Fourier coefficients of $r$ (up to a factor involving $2\pi$). We use similar notation for $r'$ and $r r'$. Using that $Q$ is bandlimited, we insert the expression $Q = \sum_{|l| \leq M} Q_l e^{il \cdot x}$ to obtain 
\[
\left| \int_{\Omega} e^{-ik \cdot x} Q \,dx \right|  \leq C \|\Lambda_q(\lambda)-\Lambda_{q'}(\lambda)\| + \sum_{|l| \leq M} |Q_l (r+r'+r r')_{k-l}|,
\]
where $C$ denotes a constant depending on $\Omega$, $n$ and $E$ that may change from line to line. On the right, we use Cauchy-Schwarz to obtain 
\[
\sum_{|l| \leq M} |Q_l r_{k-l}| \leq (\sup_{|l| \leq M} |Q_l|) \sum_{|l| \leq M} |r_{k-l}| \leq C M^{n/2} (\sup_{|l| \leq M} |Q_l|) (\sum_l |r_{k-l}|^2)^{1/2}.
\]
By the Parseval identity, the last sum is $C_n \norm{r}_{L^2(\Omega)} \leq C \lambda^{-1}$.
A similar argument applies for the term with $r'$. For the term with $r r'$, we use the simple estimate $|(r r')_m| \leq \norm{r r'}_{L^1}$ and Cauchy-Schwarz to obtain  
\begin{align*}
| \sum_{|l| \leq M} Q_l (r r')_{k-l} | &\leq (\sup_{|l| \leq M} |Q_l|) \sum_{|l| \leq M} |(r r')_{k-l}| \leq (\sup_{|l| \leq M} |Q_l|) \sum_{|l| \leq M} \norm{r r'}_{L^1} \\
 &\leq C M^n (\sup_{|l| \leq M} |Q_l|) \norm{r}_{L^2} \norm{r'}_{L^2}.
\end{align*}
The $L^2$ estimates for $r$ and $r'$ imply that 
\[
| \sum_{|l| \leq M} Q_l (r r')_{k-l} | \leq C M^n (\sup_{|l| \leq M} |Q_l|) \lambda^{-2}.
\]
Combining these facts, we get 
\[
\left| \int_{\Omega} e^{-ik \cdot x} Q \,dx \right|  \leq C \|\Lambda_q(\lambda)-\Lambda_{q'}(\lambda)\| + C (\sup_{|l| \leq M} |Q_l|) (M^{n/2} \lambda^{-1} + M^n \lambda^{-2}).
\]

To study the left hand side, we note using \eqref{eq:bandlimited_def} that 
\[
\int_{\Omega} e^{-ik \cdot x} Q \,dx = \sum_{|l| \leq M} T_{k,l} Q_l
\]
where 
\[
T_{k,l} = \int_{\Omega} e^{-i(k-l) \cdot x} \,dx.
\]
Now 
\[
\sum_{|k|, |l| \leq M} T_{k,l} c_k \bar{c}_l = \int_{\Omega} \left| \sum_{|k| \leq M} c_k e^{-ik \cdot x} \right|^2 \,dx \geq 0.
\]
Moreover, one has equality if and only if $\sum_{|k| \leq M} c_k e^{-ik \cdot x} = 0$ in $\Omega$, which is equivalent with $c_k=0$ for all $|k| \leq M$ by real-analyticity. Thus the matrix $(T_{k,l})_{|k|, |l| \leq M}$ is positive definite, hence invertible.

We have thus proved that 
\[
\norm{T \vec{Q}}_{\ell^{\infty}} \leq C \|\Lambda_q(\lambda)-\Lambda_{q'}(\lambda)\| + C \norm{\vec{Q}}_{\ell^{\infty}} (M^{n/2} \lambda^{-1} + M^n \lambda^{-2}).
\]
If we assume that $\lambda \geq M^{n/2}$, this becomes 
\[
\norm{T \vec{Q}}_{\ell^{\infty}} \leq C \|\Lambda_q(\lambda)-\Lambda_{q'}(\lambda)\| + C \norm{\vec{Q}}_{\ell^{\infty}} M^{n/2} \lambda^{-1}
\]
which is the first claim in the statement. The remaining claims follow since $T$ is invertible.
\end{proof}

\section{Real-analytic difference} \label{sec_real_analytic}

In this section, we prove Theorem~\ref{thm:analytic_case}. Our approach combines 
complex geometric optics (CGO) solutions with the decay properties of Fourier modes 
of $2\pi$-periodic analytic functions. In particular, we seek to bound
\begin{equation}
\int_\Omega e^{-ik\cdot x} Q(x) \, dx, \quad \Omega=(-\pi,\pi)^n
\end{equation}
where $Q := q - q' \in C^\omega(\overline{\Omega})$ and $k \in \mathbb{Z}^n$.

The following result is classical, but we include the proof for completeness and highlight the role that periodicity plays in our argument.

\begin{lemma}\label{Fourier_restriction_decay}
Let $\Omega = (-\pi, \pi)^n$ and let $Q \in C^\omega(\overline{\Omega})$ be periodic 
(i.e., $Q$ and all its derivatives match on opposite faces of $\partial\Omega$). 
Assume $Q$ extends holomorphically to a complex neighborhood $\Omega_\varepsilon$ 
of radius $\varepsilon > 0$. Then for all $k \in \mathbb{Z}^n$,
\begin{equation}\label{eq:Fourier_mode_analytic_decay}
\left| \int_\Omega e^{-ik\cdot x} Q(x) \, dx \right| 
\leq e^2 \operatorname{Vol}(\Omega) \|Q\|_{L^\infty(\Omega_\varepsilon)} \, 
e^{-\frac{\varepsilon |k|}{e\sqrt{n}}}.
\end{equation}
\end{lemma}

\begin{proof}
We have $(-\Delta)(e^{-ik\cdot x}) = |k|^{2} e^{-ik\cdot x}$. Since $Q$ is periodic in $\Omega = (-\pi, \pi)^n$, we may consider the integral \eqref{eq:Fourier_mode_analytic_decay} to be on a torus. By integrating by parts $N=N(k)$ times, we have 
\begin{equation*}
\int_\Omega e^{-ik \cdot x} Q \, dx
= |k|^{-2N} \int_\Omega e^{-ik \cdot x} (-\Delta)^N Q \, dx. 
\end{equation*}
Here, all the boundary terms resulting from integration by parts vanish since $Q$ is periodic. Thus, it follows that 
\begin{align*}
    \Bigl|\int_\Omega e^{-ik\cdot x} Q \, dx\Bigr| 
    &\leq |k|^{-2N} n^{N} \operatorname{Vol}(\Omega) 
       \max_{|\alpha|=2N} \|D^\alpha Q\|_{L^\infty(\Omega)}.
\end{align*}
We now estimate the right-hand side via Cauchy estimates. 
The standard Cauchy estimates applied to $Q$ on $\Omega$ give
\begin{align*}
    \|D^\alpha Q\|_{L^\infty(\Omega)}
    \leq \|Q\|_{L^\infty(\Omega_\varepsilon)} 
       \varepsilon^{-|\alpha|} |\alpha|^{|\alpha|},
\end{align*}
for all $\alpha\in \mathbb{N}^n$, where $\|Q\|_{L^\infty(\Omega_\varepsilon)}$ 
is the $L^\infty$ norm of the holomorphic extension of $Q$ to a 
complex neighborhood of $\Omega$ of radius $\varepsilon$.

Substituting this into the first term, and choosing $N \geq 0$ to be an integer with $\frac{1}{2} n^{-1/2}e^{-1}\varepsilon |k| - 1 \leq N \leq \frac{1}{2} n^{-1/2}e^{-1}\varepsilon |k|$ yields the following \begin{align*}
|k|^{-2N}n^{N}\max_{|\alpha|=2N}\|D^\alpha Q\|_{L^\infty(\Omega)} &\leq |k|^{-2N}n^{N}\|Q\|_{L^\infty(\Omega_\varepsilon)} \varepsilon^{-2N}(2N)^{2N} \\
&\leq |k|^{-2N}n^{N}\|Q\|_{L^\infty(\Omega_\varepsilon)}\varepsilon^{-2N}\left(n^{-1/2}e^{-1}\varepsilon|k|\right)^{2N} \\
&= \|Q\|_{L^\infty(\Omega_\varepsilon)} e^{-2 N} \leq e^2 \|Q\|_{L^\infty(\Omega_\varepsilon)} e^{- \frac{\varepsilon|k|}{e \sqrt{n}} }. \qedhere
\end{align*}
\end{proof}

\begin{remark}
If \(Q\) is not periodic, integration by parts produces boundary terms 
that do not vanish. The boundary sum contains terms of the form 
\(|k|^{-m-1} \max_{|\alpha|=m} \|D^\alpha Q\|_{L^\infty(\partial\Omega)}\) 
for \(m = 0, \dots, N-1\), which cannot all be made 
simultaneously small by a single choice of \(N\). This prevents exponential decay. 
For our purposes, this could likely be remedied by deriving a stability 
estimate for boundary determination in the analytic case, but we leave 
this to future work. 
\end{remark}
We next record an elementary statement for brevity, since it is used in the following proofs of Theorem \ref{thm:analytic_case} and \ref{thm:stretched_exponential}.

\begin{lemma}\label{annulus_decomposition}
Let $n\geq 1$, $a>0$, and $p> 0$. There exists a constant
$C=C(n,a,p)>0$ such that, for every $M\geq 1$, 
\[
\sum_{k\in \mathbb{Z}^n,\ |k|>M}
e^{-a|k|^p}
\leq C M^n e^{-aM^p}.
\]
\end{lemma}
\begin{proof}
    We begin the proof by defining the family of shells $\{A_j\}_{j\in \mathbb{N}}$ with the purpose of partitioning $\{|k|>M\}$ into shells of thickness $2^jM$. Such shells are given by 
\begin{align*}
    A_j = \{k\in \mathbb{Z}^n\mid 2^j M<|k|\leq 2^{j+1}M\}
\end{align*}
By definition, $A_j\cap A_i=\emptyset$ for $j\neq i$. Enclosing $A_j$ in an $n$-cube of width $2^{j+2}M$ results in the existence of some constant $C_1=C_1(n)>0$ such that 
\[
    |A_j|\leq C_12^{(j+1)n}M^n
\]
which results in the following chain of inequalities
\begin{align*}
        \sum_{|k|>M} e^{-a|k|^p}\leq \sum_{j=0}^\infty |A_j|e^{-a2^{pj}M^p}\leq C_1M^n\sum_{j=0}^\infty 2^{(j+1)n}e^{-a2^{pj}M^p}.
\end{align*}
It should be noted that the first inequality follows since $k\in A_j$ and thus $|k|>2^jM$. Now, we see that since $M\geq 1$ (and thus $M^p\geq 1$)
\begin{align*}
    \sum_{j=0}^\infty 2^{(j+1)n}e^{-a2^{pj}M^p} 
    &=e^{-aM^p}\sum_{j=0}^\infty2^{(j+1)n}e^{-aM^p(2^{pj}-1)}\\ &\leq e^{-aM^p}\sum_{j=0}^\infty2^{(j+1)n}e^{-a(2^{pj}-1)}
\end{align*}
The right-most sum converges and is independent of $M$ (and depends only on $a,p$ and $n$). Thus, we have the required estimate.
\end{proof}

\begin{remark}
It is likely that in the estimate of Lemma \ref{annulus_decomposition} one could replace $M^n$ by $M^{n-1}$, but we will not need this.
\end{remark}

\begin{proof}[Proof of Theorem \ref{thm:analytic_case}]
We begin by recalling the Alessandrini identity 
\begin{align*}
    ((\Lambda_q-\Lambda_{q'})(u|_{\p\Omega}), \bar{u}'|_{\p\Omega})_{\p\Omega}=\int_\Omega Quu' dx
\end{align*}
which is valid for $u,u'\in H^1(\Omega)$ solving $(-\Delta+q)u=(-\Delta+q')u'=0$ in $\Omega$. Unlike section \ref{sec_bandlimited} where real geometrical optics solutions are chosen, we will use CGO solutions with suitably chosen phases. In particular, given $k\in \mathbb{Z}^n$, we choose complex vectors $\rho,\, \rho'\in \mathbb{C}^n$ that are given by 
\begin{align*}
    \rho = \eta + i\left(-\dfrac{k}{2}+\tau \omega\right),\qquad \rho' = -\eta+i\left(-\dfrac{k}{2}-\tau \omega \right)
\end{align*}
where $\eta,\, \omega\in \mathbb{R}^n$ such that $k\perp\eta,\ \eta \perp\omega,\ k\perp \omega$ and
\[
|\omega|=1,\qquad |\eta|^2=\dfrac{|k|^2}{4}+\tau^2.
\]
According to \cite{feldman_calderproblem_nodate}, there are $H^2(\Omega)$ solutions of $(-\Delta+q)u=(-\Delta+q')u'=0$ in $\Omega$, having the form 
\begin{align*}
    u = e^{\rho\cdot x}(1+r), \qquad u' = e^{\rho'\cdot x} (1+r')
\end{align*}
with $\|r\|_{H^s(\Omega)}\leq \tau^{s-1}\|q\|_{L^2(\Omega)}$ for $s=0,1,2$ (a similar estimate holds for $r'$). Now, substituting $u$ and $u'$ into Alessandrini's identity results in 
\begin{align*}
       \left|\int_\Omega e^{-ik\cdot x}Q dx\right| \leq \|u\|_{L^2(\p\Omega)}\|u'\|_{L^2(\p\Omega)}\|\Lambda_q-\Lambda_{q'}\|_{L^2(\p\Omega)\rightarrow L^2(\p\Omega)}\\+\left|\int_\Omega e^{-ik\cdot x}Q(r+r'+rr')dx\right|.
\end{align*}
First, we look to bound $\|u\|_{L^2(\p\Omega)}$ (and $\|u'\|_{L^2(\p\Omega)}$) in Alessandrini's identity. Since $\Omega = (-\pi,\pi)^n\subset B(0,\pi\sqrt{n})$, this is given by 
\begin{align*}
    \|u\|_{L^2(\p\Omega)} &= \|e^{\rho\cdot x}(1+r)\|_{L^2(\p\Omega)} \leq e^{\pi\sqrt{n|k|^2/4+n\tau^2}} \norm{1+r}_{L^2(\p \Omega)} \\ 
    &\leq e^{\pi\sqrt{n|k|^2/4+n\tau^2}} C_{n,\Omega}(1 + \norm{r}_{H^1(\Omega)}) \\
    &\leq C_{n,\Omega, E} e^{\pi\sqrt{n|k|^2/4+n\tau^2}}.
\end{align*}
A similar estimate holds for $\|u'\|_{L^2(\p\Omega)}$. Substituting these into Alessandrini's identity yields 
\begin{multline}\label{real_analytic_stability_est1}
\left|\int_\Omega e^{-ik\cdot x}Qdx\right| \leq C_{n,\Omega, E} e^{2\pi\sqrt{n|k|^2/4+n\tau^2}}\|\Lambda_q-\Lambda_{q'}\| \\
 +\left|\int_\Omega e^{-ik\cdot x}Q(r+r'+rr')dx\right|.
\end{multline}

Next, as in the proof of Theorem~\ref{thm_main1_omega}, we estimate the remainder terms using Parseval's identity. Since $r, r' \in L^2(\Omega)$ and $rr'\in L^1(\Omega)$, their Fourier coefficients $r_k = \int_{\Omega} e^{-ik \cdot x} r \,dx$ etc are well-defined (whether or not the functions are periodic), and inserting $Q = \sum Q_l e^{il \cdot x}$ which converges by Lemma \ref{Fourier_restriction_decay} yields 
\begin{align}\label{alessandrini_remainder_bound}
\left|\int_\Omega e^{-ik\cdot x}Q(r+r'+rr')\right|&\leq \sum_{l\in \mathbb{Z}^n} |Q_l(r+r'+rr')_{k-l}|.
\end{align}
We can estimate the first (and second) term on the right-hand side as follows
\begin{align*}
\sum_{l\in \mathbb{Z}^n} |Q_lr_{k-l}|&\leq \sup_{|l|\leq M}|Q_l| \sum_{|l|\leq M} |r_{k-l}|+\sum_{|l|>M}|Q_lr_{k-l}|\\ 
&\leq CM^{n/2} \|r\|_{L^2(\Omega)}\sup_{|l|\leq M}|Q_l|+ C\|r\|_{L^2(\Omega)}\left(\sum_{|l|>M}|Q_l|^2\right)^{1/2}\\ 
&\leq C\tau^{-1}\left(M^{n/2}\sup_{|l|\leq M}|Q_l|+\left(\sum_{|l|>M}|Q_l|^2\right)^{1/2}\right).
\end{align*}
Here $C$ is a dimensional constant. For the term involving $rr'$, we estimate $|(rr')_m| \leq \|rr'\|_{L^1(\Omega)}$ and apply the Cauchy-Schwarz inequality as in the proof of Theorem~\ref{thm_main1_omega}. This gives
\begin{align*}
    \sum_{l\in \mathbb{Z}^n} |Q_l(rr')_{k-l}|\leq \sum_{l\in \mathbb{Z}^n}|Q_l|\|rr'\|_{L^1(\Omega)}\leq
 \sum_{l\in \mathbb{Z}^n}|Q_l| \|r\|_{L^2(\Omega)}\|r'\|_{L^2(\Omega)}.
\end{align*}
Thus, it follows that 
\begin{align*}
    \sum_{l\in \mathbb{Z}^n} |Q_l(rr')_{k-l}| \leq C\tau^{-2} \sum_{l\in \mathbb{Z}^n}|Q_l|\leq C\tau^{-2}\left(M^n\sup_{|l|\leq M}|Q_l|+\sum_{|l|>M}|Q_l|\right).
\end{align*}
Substituting these estimates into \eqref{alessandrini_remainder_bound} results in 
\begin{multline*}
        \left|\int_\Omega e^{-ik\cdot x}Q(r+r'+rr')dx\right| \leq C(\tau^{-1}M^{n/2}+\tau^{-2}M^n)\sup_{|l|\leq M}|Q_l|\\+C\tau^{-1}\left(\sum_{|l|>M}|Q_l|^2\right)^{1/2}+C\tau^{-2}\sum_{|l|>M}|Q_l|. 
\end{multline*}
From here, the periodicity and real-analyticity of $Q$ allows us to invoke Lemma \ref{Fourier_restriction_decay}. That is, for all $l\in \mathbb{Z}^n$, 
\begin{align*}
    |Q_l|\leq Ce^{-\sigma |l|} \quad \text{for some} \;\ \sigma>0.
\end{align*}
where $C=C(n,\Omega,E)>0$. Upon substituting this into \eqref{real_analytic_stability_est1}, we have that 
\begin{multline*}
        \left|\int_\Omega e^{-ik\cdot x}Q(r+r'+rr')dx\right| \leq C(\tau^{-1}M^{n/2}+\tau^{-2}M^n)\sup_{|l|\leq M}|Q_l|\\+C\tau^{-1}\left(\sum_{|l|>M}e^{-2\sigma|l|}\right)^{1/2}+C\tau^{-2}\left(\sum_{|l|>M}e^{-\sigma|l|}\right). 
\end{multline*}
Here, we can utilise Lemma \ref{annulus_decomposition} to estimate the final two terms on the right-hand side. Namely, we take $a=2\sigma,\ p=1$ in the first term and $a=\sigma,\ p=1$ in the second term. Thus it follows that
\begin{multline*}
        \left|\int_\Omega e^{-ik\cdot x}Q(r+r'+rr')dx\right| \leq C(\tau^{-1}M^{n/2}+\tau^{-2}M^n)\sup_{|l|\leq M}|Q_l|\\+\widetilde{C}_1\tau^{-1}M^{n/2}e^{-\sigma M}+\widetilde{C}_1\tau^{-2}M^{n}e^{-\sigma M}. 
\end{multline*}
Where $\widetilde{C}_1=\widetilde{C}_1(n,\sigma,C)>0$. Now, we choose $\tau=\kappa M^{n/2}$ where $\kappa>0$ such that 
\begin{align}\label{coupling_constant_cond}
    \kappa^{-1}+\kappa^{-2}<\dfrac{1}{2C}.
\end{align}
With this selection of $\tau$, the above estimate simplifies to 
\begin{align*}
    \left|\int_\Omega e^{-ik\cdot x}Q(r+r'+rr')dx\right| \leq \dfrac{1}{2}\sup_{|l|\leq M}|Q_l|+\widetilde{C}_1\left(\kappa^{-1}+\kappa^{-2}\right)e^{-\sigma M}.
\end{align*}
Since $M>0$, the coupling constant condition \eqref{coupling_constant_cond} implies that 
\begin{align*}
    \left|\int_\Omega e^{-ik\cdot x}Q(r+r'+rr')dx\right| \leq \dfrac{1}{2}\sup_{|l|\leq M}|Q_l|+\widetilde{C}_2e^{-\sigma M}.
\end{align*}
with $\widetilde{C}_2=\widetilde{C}_2(\widetilde{C}_1)>0$. Thus, after substituting the above estimate into \eqref{real_analytic_stability_est1}, taking the supremum over $|k|\leq M$ and absorbing the first term of the remainder into the left-hand side, we have that 
\begin{align}
    \sup_{|k|\leq M}|Q_k|\leq C e^{\gamma M^{n/2}}\|\Lambda_q-\Lambda_{q'}\|_{L^2(\p\Omega)\rightarrow L^2(\p\Omega)}+\widetilde{C}_2 e^{-\sigma M},
\end{align}
where $\gamma = 2\pi\sqrt{n\kappa^2+n/4}$ with $C=C(n,\Omega,E_1,E_2)>0$.



We then look to balance the two terms on the right-hand side. That is, if $\delta = \|\Lambda_q-\Lambda_{q'}\|_{L^2(\p\Omega)\rightarrow L^2(\p\Omega)}$, we choose $M = M(\delta)$ to satisfy
\begin{align*}
    e^{\gamma M^{n/2}} \delta = e^{-\sigma M} \iff \gamma M^{n/2}+\sigma M=\log(1/\delta).
\end{align*}
For any $\delta < 1$, the right-hand side is positive. Since $n\geq 3$, the function $f(M) = \gamma M^{n/2} + \sigma M$ is strictly increasing on $[0,\infty)$ with $f(0)=0$ and $f(M)\to\infty$ as $M\to\infty$. Hence there exists a unique solution $M = M(\delta) > 0$, and $M(\delta) \to \infty$ as $\delta \to 0$. Moreover, we have $M(\delta)^{n/2} \geq M(\delta)$ for all $\delta$ small enough. Therefore
\[
\gamma M(\delta)^{n/2} + \sigma M(\delta) \leq (\gamma + \sigma) M(\delta)^{n/2}.
\]
Combining this with the balancing condition yields the explicit bound 
\[
M(\delta)\geq \left(\frac{1}{\gamma+\sigma}\log(1/\delta)\right)^{2/n}.
\]
Thus, we have
\[
e^{-\sigma M(\delta)} \leq \exp\left(-\sigma \left(\frac{1}{\gamma+\sigma}\log(1/\delta)\right)^{2/n}\right).
\]
Substituting this into the estimate yields, for all sufficiently small $\delta$,
\begin{align*}
    \sup_{|k|\leq M(\delta)} |Q_k| \leq C \exp\left(-\sigma \left(\dfrac{1}{\gamma+\sigma}\log(1/\delta)\right)^{2/n}\right),
\end{align*}
where $C = C(\Omega, n, E_1, E_2) > 0$.
\end{proof}

\section{Potentials with super-exponential Fourier decay}\label{sec_super_exponential}

In this section, we strengthen the decay assumption on the Fourier coefficients of the difference $Q = q - q'$ compared to the real-analytic case of Section~\ref{sec_real_analytic}. We assume a stretched-exponential decay of order $n/2$, namely
\begin{equation}\label{eq:stretched_decay}
|Q_k| \leq C e^{-c |k|^{n/2}}, \qquad k \in \mathbb{Z}^n,
\end{equation}
for some $C, c > 0$. Since $n \geq 3$, this decay is super-exponential and thus stronger than the exponential decay $e^{-\sigma |k|}$ characteristic of real-analytic functions. As a consequence, we obtain H\"older stability for the low Fourier modes, improving upon the sub-H\"older estimate of Theorem~\ref{thm:analytic_case}.




\begin{proof}[Proof of Theorem \ref{thm:stretched_exponential}]
The proof follows the same lines as Theorem~\ref{thm:analytic_case}. We use CGO solutions as in that proof with complex phases depending on a parameter $\tau > 0$. From the Alessandrini identity and the bounds on the CGO solutions, we obtain for any $M \geq 1$ and any $\tau > 0$,
\begin{multline}\label{eq:stretched_master}
\sup_{|k| \leq M} |Q_k| \leq C_{n,\Omega, E} e^{2\pi\sqrt{n|k|^2/4+n\tau^2}} \delta \\ 
+ C(\tau^{-1} M^{n/2}
 + \tau^{-2} M^n) \sup_{|k| \leq M} |Q_k| 
+ R_{\text{tail}}(M),
\end{multline}
where $R_{\text{tail}}(M)$ contains the  contributions
\[
R_{\text{tail}}(M) = C\tau^{-1}\Big(\sum_{|l| > M} |Q_l|^2\Big)^{1/2} + C\tau^{-2} \sum_{|l| > M} |Q_l|,
\]
and $\delta = \|\Lambda_q - \Lambda_{q'}\|_{L^2(\partial\Omega) \to L^2(\partial\Omega)}$.

We now choose $\tau = \kappa M^{n/2}$ with $\kappa > 0$ as in Theorem~\ref{thm:analytic_case}. Then, as in that proof, the first term in \eqref{eq:stretched_master} is bounded by $C e^{\gamma M^{n/2}} \delta$ and the coefficient in the second term is $C(\kappa^{-1} + \kappa^{-2}) \leq 1/2$.

It remains to estimate $R_{\text{tail}}(M)$. Using the decay condition \eqref{eq:stretched_decay} and Lemma \ref{annulus_decomposition}, we can set $a=2c,\ p=n/2$ for the first summation and $a=c,\ p=n/2$ in the second summation in $R_{\mathrm{tail}}(M)$ to obtain 
\begin{align*}
    \sum_{|l|>M} |Q_l|^2 \leq C_1M^ne^{-2cM^{n/2}}, \quad
    \sum_{|l|>M} |Q_l| \leq C_2M^n e^{-cM^{n/2}}
\end{align*}
Substituting $\tau = \kappa M^{n/2}$, the tail contributions become
\begin{align*}
R_{\text{tail}}(M) &\leq C_1\kappa^{-1} e^{-c M^{n/2}} + C_2\kappa^{-2} e^{-c M^{n/2}} \\ &\leq C'(\kappa^{-1}+\kappa^{-2}) e^{-c M^{n/2}}
\end{align*}
where $C'=\max(C_1,C_2)>0$. Therefore
\[
R_{\text{tail}}(M) \leq C_3 e^{-c M^{n/2}}.
\]
where $C_3=C_3(C',C)>0$.
Inserting these estimates into \eqref{eq:stretched_master} and absorbing the second term into the left-hand side yields
\[
\sup_{|k| \leq M} |Q_k| \leq C e^{\gamma M^{n/2}} \delta + C_3 e^{-c M^{n/2}}.
\]

We balance the two terms on the right-hand side by choosing $M = M(\delta)$ so that
\[
e^{\gamma M^{n/2}} \delta = e^{-c M^{n/2}}.
\]
This is equivalent to
\[
e^{(\gamma + c) M^{n/2}} = \delta^{-1},
\]
which gives
\[
M^{n/2} = \frac{1}{\gamma + c} \log(1/\delta).
\]
With this choice, both terms are equal to
\[
e^{-c M^{n/2}} = \exp\left(-\frac{c}{\gamma + c} \log(1/\delta)\right) = \delta^{\frac{c}{\gamma + c}}.
\]
Thus we obtain
\[
\sup_{|k| \leq M(\delta)} |Q_k| \leq C \delta^{\frac{c}{\gamma + c}},
\]
where
\[
M(\delta) = \left( \frac{1}{\gamma + c} \log(1/\delta) \right)^{2/n}
\]
is the number of Fourier modes that can be stably recovered. This completes the proof with $\alpha = c/(\gamma + c)$.
\end{proof}

\begin{remark}We see that $M(\delta) \to \infty$ as $\delta \to 0$, meaning that the smaller the difference of the DN maps, the more Fourier coefficients of $Q$ can be H\"older stably reconstructed. 

The H\"older exponent $\alpha = c/(\gamma + c)$ in Theorem~\ref{thm:stretched_exponential} depends on the decay rate $c$ of the Fourier coefficients and on the geometric constant $\gamma$. A larger $c$ (faster decay) yields a larger $\alpha$, approaching $1$ as $c \to \infty$. Compared to Theorem~\ref{thm:analytic_case}, the super-exponential decay \eqref{eq:stretched_decay}  yields a stronger stability estimate, upgrading the sub-H\"older modulus to a genuine H\"older modulus.
\end{remark}

\begin{remark}[Bandlimited difference]
If the difference $Q = q - q'$ is $N$-bandlimited for some fixed $N \in \mathbb{N}$ as in \eqref{eq:bandlimited_def}, then the tail term $R_{\text{tail}}$ in the proof of Theorem~\ref{thm:stretched_exponential} vanishes identically. In this case, the argument simplifies and yields the Lipschitz stability estimate
\[
\sup_{|k| \leq N} |Q_k| \leq C e^{\gamma N^{n/2}} \|\Lambda_q - \Lambda_{q'}\|_{L^2(\partial\Omega) \to L^2(\partial\Omega)},
\]
where the constants $C, \gamma$ depend only on $\Omega, n$, and the a priori bounds on $\|q\|_{L^\infty}$ and $\|q'\|_{L^\infty}$. No smallness condition on $\|\Lambda_q - \Lambda_{q'}\|$ is required.

This estimate is comparable to known results such as \cite[Theorem~2]{alberti2022calderon}, with the difference that the constant here grows as $e^{cN^{n/2}}$ rather than $e^{cN^{1/2}}$, but in return only the difference $q-q'$ is assumed to be $N$-bandlimited. The individual potentials $q$ and $q'$ need not belong to any finite-dimensional subspace.
\end{remark}

\appendix

\section{Difference of DN maps on $L^2(\p \Omega)$} \label{sec_difference_dn_maps}

In this section we show that the difference of DN maps has finite norm on $L^2(\p \Omega)$. If the coefficients are smooth, this follows from the fact that $\Lambda_{q_1}(\lambda)$ and $\Lambda_{q_2}(\lambda)$ are pseudodifferential operators of order one with the same principal symbol, and hence the difference is a pseudodifferential operator of order $0$. The next result shows that the same holds with low regularity coefficients.

\begin{lemma} \label{lemma_dnmap_difference_ltwo}
Let $\Omega \subset \R^n$ be a bounded domain with smooth boundary, let $\lambda \geq 0$, and let $q_1, q_2 \in L^{\infty}(\Omega)$ be such that the DN maps $\Lambda_{q_j}(\lambda)$ are well defined. Then 
\[
\norm{\Lambda_{q_1}(\lambda) - \Lambda_{q_2}(\lambda)}_{L^2 \to L^2} \leq C \norm{q_1-q_2}_{L^{\infty}(\Omega)}.
\]
\end{lemma}
\begin{proof}
Fix $f \in H^{1/2}(\p \Omega)$ and let $u_j$ solve $(-\Delta-\lambda^2+q_j) u_j = 0$ in $\Omega$ with $u_j|_{\p \Omega} = f$. Then $w = u_1-u_2$ solves 
\[
(-\Delta-\lambda^2+q_1)w = (q_2-q_1)u_2 \text{ in $\Omega$}, \qquad w|_{\p \Omega} = 0.
\]
Note that $(\Lambda_{q_1}(\lambda) - \Lambda_{q_2}(\lambda)) f = \p_{\nu} w|_{\p \Omega}$. By trace and elliptic estimates (where we can afford to lose a bit of derivatives) 
\begin{align*}
\norm{\p_{\nu} w}_{L^2(\p \Omega)} &\lesssim \norm{w}_{H^2(\Omega)} \lesssim \norm{q_2-q_1}_{L^{\infty}(\Omega)} \norm{u_2}_{L^2(\Omega)}.
\end{align*}

It remains to show that $\norm{u_2}_{L^2(\Omega)} \lesssim \norm{f}_{L^2(\p \Omega)}$. This is a standard duality argument. Let $v \in L^2(\Omega)$, and let $\varphi \in H^2(\Omega)$ solve 
\[
(-\Delta-\lambda^2+q_2)\varphi = v \text{ in $\Omega$}, \qquad \varphi|_{\p \Omega} = 0.
\]
Then integration by parts and the equation for $u_2$ give 
\[
(u_2, v)_{L^2(\Omega)} = (u_2, (-\Delta-\lambda^2+q_2)\varphi)_{L^2(\Omega)} = -(f, \p_{\nu} \varphi)_{L^2(\p \Omega)}.
\]
By trace and elliptic estimates, 
\[
|(u_2, v)_{L^2(\Omega)}| \lesssim \norm{f}_{L^2(\p \Omega)} \norm{\varphi}_{H^2(\Omega)} \lesssim \norm{f}_{L^2(\p \Omega)} \norm{v}_{L^2(\Omega)},
\]
which gives the required result.
\end{proof}

\printbibliography 
\end{document}